\documentclass{amsart}
\makeatletter
\@namedef{subjclassname@2020}{%
  \textup{2020} Mathematics Subject Classification}
\makeatother
\usepackage{latexsym}
\usepackage{amssymb}
\usepackage{graphicx}
\usepackage{wrapfig}
\usepackage{amsthm}
\usepackage{float}
\usepackage{epsfig}
\usepackage{multirow}
\usepackage{xcolor}
\usepackage{caption}
\usepackage[toc,page]{appendix}
\usepackage{old-arrows} 
\usepackage{caption}
\usepackage{subcaption}
\usepackage{lmodern, mathtools}
\usepackage{mathtools}
\usepackage{hyperref}
\RequirePackage{color}
\usepackage{tikz}
\usepackage{amscd,enumerate}
\usepackage{amsfonts}
\hypersetup{ colorlinks,
linkcolor=blue,
filecolor=green,
urlcolor=blue,
citecolor=blue }
\usepackage{helvet}
\usepackage{bigstrut}
\usepackage{float, caption, booktabs, cellspace}
\usepackage{calrsfs}
\DeclareMathAlphabet{\pazocal}{OMS}{zplm}{m}{n}

\newcommand{\Lim}{\displaystyle\lim}

\makeatletter
\@addtoreset{subfigure}{framenumber}
\makeatother

\usepackage{etoolbox}
\AtBeginEnvironment{figure}{\setcounter{subfigure}{0}}

\newtheorem{lemma}{Lemma}[section]
\newtheorem{corollary}[lemma]{Corollary}
\newtheorem{theorem}[lemma]{Theorem}
\newtheorem{proposition}[lemma]{Proposition}
\newtheorem{remark}[lemma]{Remark}
\newtheorem{definition}[lemma]{Definition}

\newtheorem{notation}[lemma]{Notation}

\input xygraph
\input xy
\xyoption{all}

\usepackage{array}
\newcolumntype{L}[1]{>{\raggedright\let\newline\\\arraybackslash\hspace{0pt}}m{#1}}
\newcolumntype{C}[1]{>{\centering\let\newline\\\arraybackslash\hspace{0pt}}m{#1}}
\newcolumntype{R}[1]{>{\raggedleft\let\newline\\\arraybackslash\hspace{0pt}}m{#1}}

\usepackage{multicol}

\newcommand{\bK}{\mathbb{K}}
\def\l{\lambda}
\def\a{\alpha}

\def\b{\beta}
\def\g{\gamma}
\def\o{\omega}

\def\m{\mu}

\def\GL{\mathop{\hbox{\rm GL}}}

\def\B{\mathcal{B} }

\def\remove#1{}
\newcommand{\Supp}{{\rm supp}}

\newcommand{\G}{{\mathcal{G}}}

\newcommand{\N}{{\mathbb{N}}}
\newcommand{\Z}{{\mathbb{Z}}}
\newcommand{\K}{{\mathbb{K}}}
\newcommand{\KN}{{\mathbb{K}^{\times}}}
\newcommand{\Q}{{\mathbb{Q}}}
\newcommand{\R}{{\mathbb{R}}}

\newcommand{\dos}{\mathbf{II}}
\newcommand{\tres}{\mathbf{III}}

\def\ch{\mathop{\hbox{\rm char}}}

\newcommand{\D}{\mathfrak{Diag}}

\title{On simple evolution algebras of dimension two and three. Constructing simple and semisimple evolution algebras}

\author[Y. Cabrera]{Yolanda Cabrera Casado}

\author[D. Mart\'{\i}n]{Dolores Mart\'{\i}n Barquero}
\author[C. Mart\'{\i}n]{C\'andido Mart\'{\i}n Gonz\'alez}
\author[A. Tocino]{Alicia Tocino}

\address{Departamento de Matem\'atica Aplicada, E.T.S. Ingenier\'\i a Inform\'atica, Universidad de M\'alaga, Campus de Teatinos s/n. 29071 M\'alaga.   Spain. }
\email{yolandacc@uma.es}
\address{Departamento de Matem\'atica Aplicada, Escuela de Ingenier\'\i as Industriales, Universidad de M\'alaga, Campus de Teatinos s/n. 29071 M\'alaga.   Spain.}
\email{dmartin@uma.es}

\address{Departamento de \'Algebra Geometr\'{\i}a y Topolog\'{\i}a, Fa\-cultad de Ciencias, Universidad de M\'alaga, Campus de Teatinos s/n. 29071 M\'alaga. Spain.} \email{candido\_m@uma.es}

\address{Departamento de Matem\'atica Aplicada, E.T.S. Ingenier\'\i a Inform\'atica, Universidad de M\'alaga, Campus de Teatinos s/n. 29071 M\'alaga.   Spain. }
 \email{alicia.tocino@uma.es}

\thanks{ The  authors are supported by the Spanish Ministerio de Ciencia e Innovaci\'on through project  PID2019-104236GB-I00 and  by the Junta de Andaluc\'{\i}a  through projects  FQM-336 and UMA18-FEDERJA-119,  all of them with FEDER funds. 
}

\begin{document}

\begin{abstract}
This work classifies  three-dimensional simple evolution algebras over arbitrary fields. For this purpose, we use tools such as the associated directed graph, the moduli set, inductive limit group, Zariski topology and the dimension of the diagonal subspace. Explicitly, in the three-dimensional case we construct some models $_i\tres_{\l_1,\ldots,\l_n}^{p,q}$ of such algebras with $1\le i\le 4$, $\l_i\in\K^\times$, $p,q\in\N$, such that any algebra is isomorphic to
one (and only one) of the given in the models and we further investigate the isomorphic question within each one. Moreover, we show how to construct simple evolution algebras of higher order from known simple evolution algebras of smaller size.
\end{abstract}

\maketitle

\section{Introduction}
The classification of three-dimensional evolution algebras over fields (under some hypothesis on the ground field) is achieved in \cite{CSV2}. In particular, the classification of evolution algebras of dimension three over the real field is obtained in \cite{BMV} and over the complex field in the paper \cite{IM}. For more information on how advances in the classification of this type of algebras have emerged, one can see \cite[Section 4.7]{ceballos}. The number of cases arising in the classification is so high that it is worth re-thinking the strategy and  ample the classification to general fields. We have divided the job into two tasks. The first one is the classification of non simple algebras, which is done in  \cite{CCGMM}. Having completed the classification of non simple evolution algebras of dimension three over arbitrary fields in the aforementioned  reference, we aboard in this paper the corresponding classification in the case of simple algebras.
The novelty in our study is the use of some tools like the diagonal subspace and the moduli sets, which systematize the organization of different classes of simple algebras into families which are parametrized by suitable sets, usually orbit sets under the action of appropriate groups.  Roughly speaking, a moduli set for a class of algebras $C$ (over the same ground field) is a set $S$ which parametrizes the isomorphic classes of elements in $C$. So, we can say that the isomorphic classes of elements of $C$ are in one-to-one correspondence with the elements of the moduli set. If we represent by $\hbox{iso}(C)$ the set whose elements are isomorphic classes of algebras of $C$, then there will be a bijection $\omega\colon \hbox{iso}(C)\to S$
in such a way that from each 
$s\in S$ we can construct the multiplication table of any $A\in\omega^{-1}(s)$. The moduli set may result from having some extra algebraic or geometric structure. For instance, there is a family of simple three-dimensional algebras $A_\l$ depending on a nonzero $\l\in\K^\times $ whose isomorphism condition is
$A_\l\cong A_\m$ if and only if $\m=\l r^7$ for some $r\in\K^\times$. So, a moduli set for this class of algebras is the quotient group $G=\K^\times/(\K^\times)^{[7]}$ where
$(\K^\times)^{[7]}:=\{r^7\colon r\in \K^\times\}$. In this case, the moduli set is a group and the isomorphism classes of such algebras are in one-to-one correspondence with $G$.
To give an example of a moduli set with a geometric flavour,
recall that a bundle in the category of sets is a surjective map $\pi\colon E\to B$ where $B$ is the base set and $E$ the total set. The sets $\pi^{-1}(b)$ are called the fibers and, of course, we have $E=\sqcup_{b\in B} \pi^{-1}(b)$.
There is another class $C'$ of three-dimensional simple algebras $B_{a,b}$ depending on $a,b\in\K^\times$ such that
$B_{a,b}\cong B_{x,y}$ if and only if there is some
$t\in\K^\times$ with $x=t^3 a$ and $y=t^7 b$. Then the parametrized curve $c_{a,b}:=\footnotesize\begin{cases}x=t^3 a\\ y=t^7 b\end{cases}$
contains all the points $(x,y)$ such that $A_{a,b}\cong A_{x,y}$. Furthermore, the set $(\K^\times\times\K^\times)/\Delta_{3,7}$, whose elements are the different curves $c_{a,b}$, is a moduli set for the class $C'$. In this case, there is a bundle (in the category of sets) $C'\to (\K^\times\times\K^\times)/\Delta_{3,7}$ such that $B_{a,b}\mapsto c_{a,b}$. This moduli set set is the base set of the bundle whose total set is $C'$ and the fibers are the isomorphic classes of $C$.

Thus, this work aims to classify simple three-dimensional evolution algebras over an arbitrary field of scalars, giving explicit constructions of moduli set for the different classes of algebras. To do the classification over  arbitrary fields, it has become necessary to use the Fröbenius normal form in just one case.

The paper is structured in the following way. In the section of preliminaries, we recall all the definitions and properties that will be used in the task of  classification. In Section \ref{simple}, we introduce the concept of diagonal subspace for $\K$-evolution algebras with a unique natural basis. It will be used to classify evolution algebras in terms of the dimension of their diagonal subspace.
In the case of two-dimensional evolution algebras, we  have three types of simple evolution algebras, that are collected in Theorem \ref{teorema1}. In the case of three-dimensional simple evolution algebras, we have twenty-seven types, which are reflected in the main Theorem \ref{teorema2}.  In the last section, we obtain new simple evolution algebras of finite dimension as the tensor product of two simple evolution algebras (see Theorem \ref{cor:strong}) in terms of the categorical product of their associated directed graphs using \cite{McANDREW}. Another way of constructing new simple evolution algebras is described in Remark \ref{nuevas}. Moreover, if $A$ is a finite evolution algebra with an ideal $I$ of dimension $1$, we prove that $A/I$ is not simple (see Theorem \ref{cociente}).

\section{Preliminaries}
A \emph{directed graph} is a $4$-tuple $E=(E^0, E^1, r_E, s_E)$ 
consisting of two disjoint sets $E^0$, $E^1$ and two maps
$r_E, s_E: E^1 \to E^0$. The elements of $E^0$ are called \emph{vertices} and the elements of 
$E^1$ are called \emph{edges} of $E$.  Further, for $e\in E^1$, $r_E(e)$ and $s_E(e)$ are 
 the \emph{range} and the \emph{source} of $e$, respectively. 
If there is no confusion with respect to the graph we are considering, we simply write $r(e)$ and $s(e)$. A \emph{path} $\mu$ of length $m$ is a finite chain of edges $\mu=f_1\ldots f_m$ such that $r(f_i)=s(f_{i+1})$ for $i=1,\ldots,m-1$. We denote by $s(\mu):=s(f_1)$ the source of $\mu$ and $r(\mu):=r(f_m)$ the range of $\mu$. If $r(\mu)=s(\mu)$,
 then $\mu$ is called a \emph{closed path}. If $E$ is a directed graph and $S\subset E^0$, then denote by $T(S)$ the tree of $S$ where $$T(S)=\{v\in E^0 \colon \text{ exist }\lambda \in \text{Path}(E) \text{ and } u \in S \text{ with } s(\lambda)=u, r(\lambda)=v\}.$$
In this paper, when we say graph we mean directed graph. An \emph{evolution algebra} over a field $\bK$ is a $\bK$-algebra $A$ which has a basis $\B=\{e_i\}_{i\in \Lambda}$ such that $e_ie_j=0$ for every $i, j \in \Lambda$ with $i\neq j$. Such a basis is called a \emph{natural basis}.
From now on, all the evolution algebras we will consider will be finite-dimensional, and  $\Lambda$ will denote a finite set $\{1, \dots, n\}$.

Let $A$ be an evolution algebra with a natural basis $\B=\{e_i\}_{i\in \Lambda}$.
Denote by $M_\B=(\omega_{ij})$ the \emph{structure matrix} of $A$ relative to $\B$, i.e., $e_i^2 = \sum_{j\in \Lambda} \omega_{ji}e_j$. In the particular case that $\det(M_\B) \neq 0$, we say that the evolution algebra is \emph{perfect}. Let $u=\sum_{i\in \Lambda}\alpha_ie_i$ be an element of $A$. Recall that 
he \emph{support of} $u$ \emph{relative to} $B$, denoted $\Supp_B(u)$, is defined as the set $\Supp_B(u)=\{i\in \Lambda\ \vert \  \alpha_i \neq 0\}$.

We recall the definition of the directed graph associated to an evolution algebra $A$ relative to a natural basis $\B=\{e_1,\ldots,e_n\}$. We draw an edge from the vertex labelled $i$ to the one labelled $j$ if the $j$-th coordinate of $e_i^2$ relative to $\B$ is nonzero. An interesting property is that 
\begin{equation}\label{first}
I:=\bigoplus_{i\in T(S)}\K e_i\triangleleft A\end{equation}
for any $S\subset E^0$. To prove this, we only need to check that $e_i^2\in I$ whenever $i\in T(S)$: but $e_i^2=\sum_j\o_{ji} e_j$ and if $\o_{ji}\ne 0$ we have 
$j\in T(i)\subset T(S)$.
One of the ideas in this work is to construct three-dimensional simple evolution algebras from algebraic constructions starting from two-dimensional evolution algebras.
We will need several algebraic tools.

It is known that for perfect evolution algebras, the number of edges in its directed graph does not depend on the natural basis chosen to get the graph. Also, the number of loops (more precisely: edges whose source and range agree) is an invariant (see \cite[Corollary 4.5]{EL1}). Moreover, recall that non-isomorphic graphs correspond to non-isomorphic evolution algebras for perfect evolution algebras.

\begin{definition}\rm
Let $A$ be a perfect evolution algebra with structure matrix $(\o_{ij})_{i,j=1}^n$.   The cardinal of the set
$\{i \ \colon \o_{ii}\ne 0\}$, which does not depend on the chosen natural basis, will be called the {\it  $l$-number of} $A$ and denoted $l(A)$. The cardinal of $\{(i,j)\colon \o_{ij}\ne 0\}$, which is another invariant, will be called the 
{\it $e$-number of} $A$ (denoted $e(A)$).
\end{definition}

Given $n\in\N$ with $n\ge 1$, and a field $\K$, we will use the notation $G_n(\K)$ for the multiplicative group $G_n(\K)=\K^\times/(\K^\times)^{[n]}$ where $(\K^\times)^{[n]}:=\{x^n\colon x\in \K^\times\}$.  
If the allusion to the field $\K$ is not necessary, we shall use $G_n$ instead of $G_n(\K)$. If $n\ge 2$, then 
there is a group action of $\Z_q$ on $G_n$ which we can describe in the following way: consider $\Z_q$ with multiplicative notation, that is $\Z_q=\{1,\pi,\ldots,\pi^{q-1}\}$ where $\pi^q=1$. Then define $\Z_q\times G_n\to G_n$ by $\pi^m \bar k=\bar k^{j^m}$ where $j$ is a fixed integer satisfying $j \neq 1$ and $j^q\equiv 1 (\hbox{mod } n)$. To prove that the above map is indeed a group action, consider $$\pi^m(\pi^l \bar k)=\pi
^m(\overline{k}^{j^l})=\overline{k}^{j^{l+m}}=\pi^{m+l}\bar k,\quad k\in\K^\times.$$

\begin{definition}\label{moduli}\rm
Let $\G$ be a group and $X$ an $\G$-set. That is, $X$ is a set, and there is a group action $\G\times X\to X$, $(m,x)\mapsto mx$. Then we say that the couple $(\G,X)$ is a {\it moduli set} for a class $C$ of algebras over the field $\K$ if there is a one-to-one correspondence between the isomorphic classes of algebras in $C$ and the set of orbits $X/\G$. Precisely, the \emph{set of orbits of a moduli set} $(\G,X)$ is defined as the set of orbits $X/\G$.

If $(\G,X)$ is a $\G$-set and $(\G',X')$ a $\G'$-set, we can define a homomorphism $f\colon (\G,X)\to (\G',X')$ as a couple $f=(f_1,f_2)$ where $f_1\colon \G\to \G'$ is a group homomorphism and $f_2\colon X\to X'$ is a map satisfying $f_2(gx)=f_1(g)f_2(x)$ for any $g\in\G$ and $x\in X$. This allows a natural notion of isomorphism. If $(\G,X)$ is a moduli set for a class of algebras $C$ and there is an isomorphism $(\G,X)\cong(\G',X')$, then $(\G',X')$ is again a moduli set for $C$. This allows a kind of simplification for certain moduli sets. An example of this situation arises when $\G$ and $\G'$ are conjugated subgroups of $\GL_n(\K)$ and we have an action
$\GL_n(\K)\times \K^n\to \K^n$ which induces  an action
$\G\times X\to X$ where $X\subset\K^n$. Let us denote the action as $g\cdot x$ for $g\in\G$ and $x\in X$. If $\G'=p\G p^{-1}$ for a fixed $p\in\GL_n(\K)$, then we can define a new action $\G'\times X'\to X'$ where $X':=p\cdot X$ denoted by $g'\bullet x'$. The new action is given by
$(pgp^{-1})\bullet (p\cdot x):=p\cdot(g\cdot x)$ for any $g\in\G$ and $x\in X$.
Furthermore, the maps $f_1\colon \G\to\G'$ given by $f_1(g)=pgp^{-1}$ and $f_2\colon X\to X'$ such that $f_2(x):=p\cdot x$ give rise to an isomorphism $(f_1,f_2)\colon (\G,X)\cong(\G',X')$. Thus, if $(\G,X)$ is a moduli set for $C$, then so is $(\G',X')$. A particular case is the following: assume that $\G$ is the subgroup of $\GL_n(\K)$ generated by a matrix $m$ and consider an action $\G\times X\to X$ for some subset $X\subset\K^n$. If $m'$ is some canonical form associated to $m$ (Jordan normal  form or alternatively, Fröbenius normal form) so that $m'=pmp^{-1}$
for some invertible matrix $p\in\GL_n(\K)$, then the subgroup $\G'$ of $\GL_n(\K)$ generated by $m'$ induces a new action on $X'=p\cdot X$ as above. So the isomorphism $(\G,X)\cong (\G',X')$ tells us that the moduli set $(\G,X)$ can be replaced with the new one $(\G',X')$ and if $(\G,X)$ classifies the algebras of $C$, then also $(\G',X')$ classifies that class of algebras. 

\end{definition}

\begin{notation}\label{mrqpp}\rm
Let $\K$ be a field and $G_n=G_n(\K)$ the group defined above. Then if $j^q\equiv 1 (\hbox{mod } n)$, we denote by 
$\Omega_{q,n}$ the couple  $\Omega_{q,n}=(\Z_q,G_n(\K))$ with the action $\pi^l \bar k:=\overline{k}^{j^l}$
for any $k\in\K^{\times}$. 
\end{notation}

For example, we will use $(\Z_2,G_3)$ as a moduli set. Explicitly, the action is $\Z_2\times G_3\to G_3$ where $\pi\bar k=\overline{k}^2$. Thus, $\bar \m$ and $\bar \l$ are in the same orbit if and only if there is some $k\in\K^{\times}$ such that $\m=k^3\l$ or $\m=k^3\l^2$ or $\l=k^3\m^2$. Another example of this moduli set is $(\Z_2,G_3({\mathbb F}_4))$.

In this case, $({\mathbb F}_4^\times)^{[3]}$ is trivial so that $G_3({\mathbb F}_4)={\mathbb F}_4^\times$. 
The action $\Z_2\times {\mathbb F}_4^{\times}\to {\mathbb F}_4^{\times}$ is given by 
$\pi\a=\b$ and $\pi\b=\a$. Thus, the orbit set ${\mathbb F}_4^\times/\Z_2$ has two elements: the orbit of $1$ which is a singleton and the orbit of $\a$ which is $\{\a,\b\}$.

If $G$ is an abelian group and we consider the direct system $G\to G\to\cdots\to G\to\cdots$ where every homomorphism is given by $g\mapsto g^m$ (for  fixed $m$), then we can construct the inductive limit group $\Lim_{\to m} G $ whose elements are the equivalence classes $[g]=\{g^{m^k}\colon k\in\N\}$. Thus $[g]=[h]$ if and only if $g^{m^r}=h^{m^s}$ for some naturals $r$ and $s$. 
We will apply this construction when $G=G_n(\K)$. 
The canonical homomorphism $\K^{\times} \to G_n(\K)\to \Lim_{\to m}G_n(\K)$ allows us to say that
two elements $\l,\m\in\K^\times$ have the same image in $\Lim_{\to m}G_n(\K)$ when 
$\l^{m^r}=k^n\m^{m^s}$ for $r,s\in\N$ and $k\in\K^\times$.
For instance, in the moduli set $(\Z_2,G_3)$ whose set of orbits is $G_3/\Z_2$ we have  $G_3/\Z_2=\Lim_{\to 2}G_3$ and two elements $\l,\m\in\K^\times$ have the same image in $\Lim_{\to 2}G_3(\K)$ if and only if $\l^{2^r}=k^3\m^{2^s}$.

\begin{definition} \rm 
Let $\K$ be a field and  $n_1,\ldots,n_q \in \Z$. We define a map $\phi: \K^{\times} \to (\K^{\times})^{[n_1]} \times \ldots \times (\K^{\times})^{[n_q]}  $ such that $\phi(\lambda)=(\lambda^{n_1},\ldots,\lambda^{n_q})$. Observe that $\phi$ is a group homomorphism. We denote by $\Delta_{n_1,\ldots,n_q}=\phi(\K^{\times})$.
\end{definition}

\begin{definition}\rm \label{action1}
We define an action $\GL_n(\K)\times(\K^\times)^n\to (\K^\times)^n$ such that 
$$(m_{i,j})\cdot (v_i):=(\prod_j v_j^{m_{1,j}},\ldots,\prod_j v_j^{m_{n,j}})^T.$$
Then, for any subgroup $S$ of $\GL_n(\K)$ we can restrict to an action $S\times(\K^\times)^n\to (\K^\times)^n$. Denote the orbit of an element $v\in(\K^\times)^n$ by $[v]$.
For a fixed matrix $M=(m_{i,j}) \in \GL_n(\K)$, we can consider the subgroup $S=\{M^k\colon k\in\Z\}$ of $\GL_n(\K)$ and the action $S\times (\K^\times)^n\to (\K^\times)^n$ as before. This gives rise to the orbit space $(\K^\times)^n/S$, which will depend on $S$ (and of course of $\K$). In some cases, there can be a power of $M$ which equals the identity. Then, the group $S$ will be finite. On the contrary, we will have $S\cong \Z$.

\end{definition}

\begin{notation}\rm
Let $\K$ be an arbitrary field. If $p \in \K[x_1,\ldots,x_n]$, we  denote by $D(p):=\{(\l_1,\ldots \l_n) \in \K^{n} \colon p(\l_1,\ldots \l_n) \neq 0\}$. Observe that $D(p)$ is a basic open set in the Zariski topology of $\K^{n}$.
\end{notation}

\section{Classification of two and three dimensional simple evolution algebras}\label{simple}

 A simple $\K$-algebra $A$ is an algebra such that $A^2\ne 0$ and its only ideals are $0$ and $A$.
We recall also that a directed graph $E$ is said to be \emph{strongly connected} if given any two  different vertices $u$ y $v$ in $E^0$, there exists a path $\mu$ such that $s(\mu)=u$ and $r(\mu)=v$. We know that a perfect finite evolution algebra is simple if and only if its associated directed graph is strongly connected (see \cite[Proposition 2.7]{CKS1}).

\begin{remark} \label{yaemp}\rm 
It is easy to check that a directed graph is strongly connected if and only if it has a  closed path containing all the vertices. So, a perfect evolution algebra whose associated directed graph has a closed path containing all the vertices is simple. In particular, if $A$ is a  perfect evolution algebra with structure matrix $(\omega_{ij})$ satisfying  $\omega_{ij}\ne 0$ for  $i \neq j$, then $A$ is simple.
\end{remark}

Assume that $A$ is an evolution $\K$-algebra such that any two natural basis $\B_1$ and $\B_2$ are related in the sense that for any $e\in \B_1$, there is a nonzero scalar $k \in \K$ and an $f\in \B_2$ such that $f=k e$. When this happen, we say that \textit{$A$ has a unique natural basis} (see \cite[Definition 2.1]{boudi20}). Note that $\sum_{e\in \B_1}(eA)e=\sum_{f\in \B_2}(fA)f$, then we have the following definition.

\begin{definition}\rm
In the previous conditions, we define the \emph{diagonal subspace}  given by $$\D(A):=\sum_{e\in \B} (eA)e.$$
\end{definition}

Observe that, for perfect evolution $\K$-algebras, the above definition applies (and, in particular, for simple algebras). Furthermore, if $f\colon A_1\to A_2$ is an isomorphism of algebras and $A_1$ is in the above conditions, $f(\D(A_1))=\D(A_2)$. The diagonal subspace is then an invariant for simple evolution $\K$-algebras,  and we can classify the simple three-dimensional algebras $A$ into $4$ disjoint classes depending on $\dim(\D(A))$. Note also that 
$\D(A)=\oplus_{i}\K e_i^2$ (extended to those indices $i$ such that $\omega_{ii}\ne 0$). In other words, $\dim(\D(A))=l(A)$   (the number of loops in the directed graph associated).

\begin{notation}\rm
We will consider the notation $\dos^{l(A),e(A)}_{\Gamma}$ for $A$ a two-dimensional evolution algebra, where $\Gamma$ is the set of nonzero parameters that appears in its corresponding structure matrix. Similarly, we will consider $_n\tres^{l(A),e(A)}_{\Gamma}$ for $A$ a three-dimensional evolution algebra where $n \in \N^*$ helps us to distinguish the cases in which the evolution algebras  have the same $l(A)$, $e(A)$ and $\Gamma$.

\end{notation}

\subsection{Two-dimensional case}
So, for instance, the case of two-dimensional simple evolution algebras $A$ will be useful to illustrate the above ideas. Since the graph must contain a cycle visiting the two vertices, the graph  will contain the subgraph:
$$\xymatrix{{\bullet}^{1} \ar@/^0.5pc/ [r]   & {\bullet}^{2} \ar@/^0.5pc/[l]}$$
\noindent so that $\dim(\D(A))=l(A)=0,1,2$.
\begin{enumerate}
\item If $\dim(\D(A))=0$ (implying $e(A)=2$),  we have $e_1^2=\a e_2$ and $e_2^2=\b e_1$. Without loss of generality, we may assume $e_1^2=e_2$ and $e_2^2=\l e_1$ with $\l\in\K^\times$. Thus, the structure matrix relative to $\{e_1,e_2\}$ is 
$\tiny\begin{pmatrix} 0 & \l\cr 1 & 0\end{pmatrix}$. This algebra is $\dos_{\l}^{0,2}$.
Now, we focus on natural basis whose structure matrix is of the kind above (zeros in the diagonal, $1$ in the $(2,1)$-entry and a nonzero element in the $(1,2)$-entry). We can see that any other natural basis of this kind is either $\{k e_1,k^2 e_2\}$ or  $\{k e_2,k^2\l e_1\}$ with $k\in\K^\times$. So, here we have an action of the group $\K^\times\times\Z_2$ on the set of natural bases. The structure matrix relative to $\{k e_1,k^2 e_2\}$ is $\tiny\begin{pmatrix} 0 & k^3\l\cr 1 & 0\end{pmatrix}$ while the corresponding one relative to 
$\{k e_2,k^2\l e_1\}$ is $\tiny\begin{pmatrix} 0 & k^3\l^2\cr 1 & 0\end{pmatrix}$. 
Therefore, the algebras $\dos_{\l}^{0,2}$ and $\dos_{\m}^{0,2}$ with structure matrices $\tiny\begin{pmatrix} 0 & \l\cr 1 & 0\end{pmatrix}$ and
$\tiny\begin{pmatrix} 0 & \m\cr 1 & 0\end{pmatrix}$ are isomorphic if and only if there is a
$k\in\K^\times$ such that $\m^{2^{r}}=k^3\l^{2^{s}}$ for some $r, s \in \N$. Thus, the isomorphism classes of algebras with structure matrix of type 
$\tiny\begin{pmatrix} 0 & \l\cr 1 & 0\end{pmatrix}$ are classified by the moduli set $\Omega_{2,3}=(\Z_2,G_3)$ (see Notation \ref{mrqpp}).  Equivalently, $\dos_{\l}^{0,2}\cong \dos_{\m}^{0,2}$ if and only if $\l$ and $\m$ have the same image in $\Lim_{\to 2}G_3(\K)$.

\item If $\dim(\D(A))=1$, then $e(A)=3$. We may choose our natural basis to have structure matrix $\tiny\begin{pmatrix}1 & \l\cr 1 & 0\end{pmatrix}$ with $\l\in\K^{\times}$. This algebra is $\dos_{\l}^{1,3}$.
\bigskip
$$
\tiny\xymatrix{\bullet^{1} \ar@(dl,ul) \ar@/^.5pc/[r]  & \bullet^{2} \ar@/^.5pc/[l]}$$

\bigskip

It is easy to check that if $\dos_{\m}^{1,3}\cong \dos_{\l}^{1,3}$, then  $\m=\l$. Thus, the isomorphic classes of algebras of this kind are in one-to-one correspondence with the set $\K^\times$. We may consider that the group acting is the trivial one. So, the moduli set is $(1,\K^\times)$.
\item If $\dim(\D(A))=2$ (so $e(A)=4$), we can choose the matrix relative to a natural basis of the form $\tiny\begin{pmatrix}1 & \l\cr \mu & 1\end{pmatrix}$ with 
$\l,\m \in\K^\times$ and $\l\m\ne 1$ for $\dos_{\l,\m}^{2,3}$.  
\bigskip
$$
\tiny\xymatrix{\bullet^{1} \ar@(dl,ul) \ar@/^.5pc/[r]  & \bullet^{2} \ar@(dr,ur)\ar@/^.5pc/[l]}$$

\bigskip
If  $\dos_{\l,\m}^{2,3}\cong \dos_{\l',\m'}^{2,3}$ then we can see that $(\l',\m')=(\l,\m)$ or 
$(\l',\m')=(\m,\l)$.
We can identify the set of matrices $\tiny\begin{pmatrix}1 & \l\cr \mu & 1\end{pmatrix}$ with $\l,\m\ne 0$, $\l\m\ne 1$ with
the subset $P$ of the two-dimensional affine space $A_2(\K)$ of those points $(x,y)$ such that $x,y\ne 0$, $xy\ne 1$. So $P$ is the Zariski open subset of the affine space obtained removing both axis $x=0$, $y=0$ and the points of the \lq\lq hyperbola\rq\rq\ $xy=1$.
We have an action of the group $\Z_2$ on $P$
such that $$
0\cdot (\l,\m):=(\l,\m) \hbox{ and } 
1\cdot (\l,\m):=(\m,\l).$$
In this case, the isomorphic classes of algebras are classified by the moduli set $(\Z_2,P)$. In fact, $P=(\K^\times)^2\cap D(\l\m-1)$.
\end{enumerate}
We can summarize the results  in the following table and theorem:
\begin{table}[H]
\renewcommand{\arraystretch}{2}
\begin{tabular}{|c|C{3cm}|c|c|c|}
\hline 
  Type algebra  & Graph & Structure matrix& Moduli set  & Orbit set\cr 
\hline
\multirow{2}{*}{$\dos_{\l}^{0,2}$}  &
\multirow{2}{*}{$\tiny\xymatrix{\bullet^{1} \ar@/^0.5pc/ [r]   & \bullet^{2} \ar@/^0.5pc/[l]}$}
& \multirow{2}{*}{$\tiny\begin{pmatrix} 0 & \l\cr 1 & 0\end{pmatrix}$}& \multirow{2}{*}{$\Omega_{2,3}$} & \multirow{2}{*}{$G_3/\Z_2\cong \Lim_{\to 2}G_3(\K)$} \cr  
\multirow{3}{*}{$\dos_{\l}^{1,3}$}  & \multirow{3}{*}{$
\tiny\xymatrix{\bullet^{1} \ar@(dl,ul) \ar@/^.5pc/[r]  & \bullet^{2} \ar@/^.5pc/[l]}$}&  \multirow{3}{*}{$\tiny\begin{pmatrix}1 & \l\cr 1 & 0\end{pmatrix}$} & \multirow{3}{*}{$(1,\K^\times)$} & \multirow{3}{*}{$\K^\times$} \cr 
 \multirow{4}{*}{$\begin{array}{c}
 \dos_{\l,\m}^{2,4} 
 \end{array}$} & \multirow{4}{*}{$
\tiny\xymatrix{\bullet^{1} \ar@(dl,ul) \ar@/^.5pc/[r]  & \bullet^{2} \ar@(dr,ur)\ar@/^.5pc/[l]}$} &  \multirow{4}{*}{$\tiny\begin{pmatrix}1 & \l\cr \mu & 1\end{pmatrix}$} & \multirow{4}{*}{$(\Z_2,(\K^\times)^2\cap D(\l\m-1))$} & \multirow{4}{*}{$\frac{(\K^\times)^2\cap D(\l\m-1)}{\Z_2}$} \cr 
 &&&&\\
  &&&&\\
 \hline
\end{tabular}
\caption{\footnotesize Moduli sets of simple two-dimensional evolution algebras.}\label{table:1}
\end{table}

\begin{theorem}\label{teorema1}
Let $A$ be a two-dimensional simple evolution $\K$-algebra. Then, $A$ is isomorphic to one and only one algebra of  type $\dos_{\l}^{0,2}$, $\dos_{\l}^{1,3}$ or $\dos_{\l,\mu}^{2,4}$ (see Table \ref{table:1}).
\end{theorem}

\subsection{Three-dimensional case}
Consider a simple evolution algebra $A$ with $\dim(A)=3$.
From Remark \ref{yaemp}, we see that there is a  closed path going through the three vertices. 
Thus, $\dim(\D(A))=l(A)=0,1,2,3$.  
\subsubsection{$\boldsymbol{\dim(\D(A))=0}$}
Then $e(A)=3,4,5,6$, so we can consider the following cases:
\begin{enumerate}
    \item If $e(A)=3$ the associated directed graph is   
    \begin{center}
    
\resizebox{2.5cm}{.25cm}{
\xymatrix{ & \bullet^{2}  \ar@/^.4pc/[dr] & \cr
            \bullet^{1}\ar@/^.4pc/[ur]   & &\bullet^{3} \ar@/^.4pc/[ll]}}
 \end{center}            
             \bigskip

Without loss of generality, the structure matrix relative to a natural basis of the evolution algebra can be taken to be 
${\tiny\begin{pmatrix} 0 & 0 &\l\cr 1 & 0 & 0\cr 0 & 1 & 0\end{pmatrix}}$. If  ${\tiny\begin{pmatrix} 0 & 0 &\mu \cr 1 & 0 & 0\cr 0 & 1 & 0\end{pmatrix}}$ 
is the structure matrix relative to another natural basis, we get $\l^{2^{r}}= k^7\mu^{2^{s}}$  for some $k\in\K^\times$ and $r,s \in \N$. Thus, the isomorphism classes 
of these algebras are in one-to-one correspondence with the orbits of the group $G_7(\K)$ modulo the action of $\Z_2$ described by the moduli set $\Omega_{2,7}=(\Z_2,G_7(\K))$. That is, $\tres^{0,3}_{\l}\cong \tres^{0,3}_{\m}$ if and only if $\l$ and $\m$ have the same image in $\Lim_{\to 2}G_7(\K)$. If $\root 7 \of \lambda\in\K$, then one can see that the image of $\l$ in $\Lim_{\to 2}G_7(\K)$ is the same as the image of $1$ hence $\tres^{0,3}_{\l}\cong \tres^{0,3}_{1}$. This happens, for instance, if $\K=\R$ or $\mathbb C$. However, if $\K=\Q$ we have $\tres^{0,3}_{2}\not\cong \tres^{0,3}_{3}$.
\item If $e(A)=4$, the graph of $A$ has two possibilities:
\begin{center}
    \resizebox{6cm}{.25cm}{
\xymatrix{ & \bullet^{2} \ar@/^.5pc/[dl] \ar@/^.4pc/[dr] & \cr
            \bullet^{1} \ar@/^.4pc/[ur]   & &\bullet^{3} \ar@/^.4pc/[ll]}
            \hskip2cm \xymatrix{ & \bullet^{2} \ar@/^.5pc/[dl] \ar@/^.4pc/[dr] & \cr
            \bullet^{1} \ar@/^.4pc/[ur]   & &\bullet^{3} \ar@/^.4pc/[ul]}}
 \end{center}            
             \bigskip

But note that the second graph does not correspond to a perfect algebra.
In the first case,  there are nonzero scalars $\l,\m$ such that the structure matrix relative to a natural basis can be chosen to be ${\tiny\begin{pmatrix} 0 & \l &\m\cr 1 & 0 & 0\cr 0 & 1 & 0\end{pmatrix}}$ and we denote this evolution algebra by $\tres_{\l,\m}^{0,4}$.   If $\tres_{\l,\m}^{0,4} \cong \tres_{\l',\m'}^{0,4}$ then we find that 
$\l'=k^3\l$ and $\m'=k^7\m$ for some $k\in\K^{\times}$. Thus, we can define
the group action of $\K^{\times}$ on the set $(\K^{\times})^2$ such that 
$k\cdot (\l,\m):=(k^3\l,k^7\m)$. The moduli set $(\K^{\times},(\K^{\times})^2)$ classifies the algebras in this item. The orbit set is the underlying set of the quotient group 
$(\K^{\times})^2/\Delta_{3,7}$ where $\Delta_{3,7}=\{(k^3,k^7): k\in \K^{\times}\}$.
\item If $e(A)=5$ the graph of $A$ is
\begin{center}
    \resizebox{2.5cm}{.25cm}{
\xymatrix{ & \bullet^{2} \ar@/^.5pc/[dl] \ar@/^.4pc/[dr] & \cr
            \bullet^{1} \ar@/^.4pc/[ur]   & &\bullet^{3} \ar@/^.4pc/[ll]\ar@/^.4pc/[ul]}}
 \end{center}\smallskip
 
\noindent
and there are nonzero scalars $\l,\m,\gamma$ such the structure matrix can be chosen to be   
${\tiny\begin{pmatrix} 0 & \l &\m\cr 1 & 0 & \gamma\cr 0 & 1 & 0\end{pmatrix}}$.  Let us denote $A$ by $\tres_{\l,\m,\gamma}^{0,5}$. If  $\tres_{\l,\m,\gamma}^{0,5} \cong \tres_{\l',\m',\gamma'}^{0,5}$, then we find that $\l'=k^3\l$,
$\m'=k^7\m$ and $\g'=k^6\g$ whence we can the define the action of
$\K^\times$ on $(\K^\times)^3$ such that $k(\l,\m,\g)=(k^3\l,k^7\m,k^6\g)$ and consider the moduli set 
$(\K^\times,(\K^\times)^3)$ which classifies the algebras in this item. The orbit set is this case is the given by the quotient group $(\K^\times)^3/\Delta_{3,7,6}$.

\item If $e(A)=6$ the graph of $A$ is
\begin{center}
    \resizebox{2.5cm}{.25cm}{
\xymatrix{ & \bullet^{2} \ar@/^.5pc/[dl] \ar@/^.4pc/[dr] & \cr
            \bullet^{1} \ar@/^.4pc/[ur]  \ar@/^.5pc/[rr] & &\bullet^{3} \ar@/^.4pc/[ll]\ar@/^.4pc/[ul]}}
 \end{center}\smallskip
and there are nonzero scalars $\l,\m,\gamma, \delta$ such the structure matrix relative to a natural basis can be chosen to be   
${\tiny\begin{pmatrix} 0 & \l &\m \\
1 & 0 &\gamma \\
\delta & 1 & 0\end{pmatrix}}$. 
We denote this evolution algebra by $\tres_{\l,\m,\gamma, \delta}^{0,6}$.
If $\tres_{\l,\m,\gamma, \delta}^{0,6} \cong \tres_{\l',\m',\gamma', \delta'}^{0,6}$ then we find that $\delta'=k^{-2}\delta$,
$\l'=k^3\l$,  $\gamma'=k^6\gamma$ and $\m'=k^7\m$ whence we can define the action of
$\K^\times$ on $(\K^\times)^4$ such that $k(\delta,\l,\gamma,\mu)=(k^{-2}\delta,k^3\l,k^6\gamma, k^7 \m)$.
This action can be restricted to an action $\K^\times\times((\K^{\times})^4\cap D(\m-\l\delta\gamma))\to ((\K^{\times})^4\cap D(\m-\l\delta\gamma))$
and consider the moduli set 
$(\K^\times,(\K^{\times})^4\cap D(\m-\l\delta\gamma))$ which classifies the algebras in this item. The orbit set in this case is the given by the quotient group $$\dfrac{(\K^\times)^4\cap D(\m-\l\delta\gamma)}{\Delta_{-2,3,6,7}}.$$

\end{enumerate}

Thus, we can collect the previous results  in Table \ref{table:2}.

\begin{table}[h]
\renewcommand{\arraystretch}{1.5}
\begin{tabular}{|c|c|c|c|c|}
\hline 
 Type algebra  & Graph & Structure Matrix & Moduli set & Orbit set\cr 
\hline
\multirow{4}{*}{$\tres^{0,3}_{\l}$} & \resizebox{2.5cm}{.25cm}{
\xymatrix{ & \bullet^{2}  \ar@/^.4pc/[dr] & \cr
            \bullet^{1}\ar@/^.4pc/[ur]   & &\bullet^{3} \ar@/^.4pc/[ll]}}  &
\multirow{4}{*}{${\tiny\begin{pmatrix} 0 & 0 &\l\cr 1 & 0 & 0\cr 0 & 1 & 0\end{pmatrix}}$}
&
\multirow{4}{*}{$(1,G_7)$} & \multirow{4}{*}{$G_7$}\cr 
\multirow{4}{*}{$\tres_{\l,\m}^{0,4}$}  &
    \resizebox{2.5cm}{.25cm}{
\xymatrix{ & \bullet^{2} \ar@/^.5pc/[dl] \ar@/^.4pc/[dr] & \cr
            \bullet^{1} \ar@/^.4pc/[ur]   & &\bullet^{3} \ar@/^.4pc/[ll]}}
& 
\multirow{4}{*}{${\tiny\begin{pmatrix} 0 & \l &\m\cr 1 & 0 & 0\cr 0 & 1 & 0\end{pmatrix}}$}
& \multirow{4}{*}{$(\K^{\times},(\K^{\times})^2)$}  &    \multirow{4}{*}{$(\K^\times)^2/\Delta_{3,7}$}\cr 
\multirow{4}{*}{ $\tres_{\l,\m,\gamma}^{0,5}$}   & 
    \resizebox{2.5cm}{.25cm}{
\xymatrix{ & \bullet^{2} \ar@/^.5pc/[dl] \ar@/^.4pc/[dr] & \cr
            \bullet^{1} \ar@/^.4pc/[ur]   & &\bullet^{3} \ar@/^.4pc/[ll]\ar@/^.4pc/[ul]}}
& 
\multirow{4}{*}{${\tiny\begin{pmatrix} 0 & \l &\m\cr 1 & 0 & \gamma\cr 0 & 1 & 0\end{pmatrix}}$} & \multirow{4}{*}{$(\K^{\times},(\K^{\times})^3)$}  & \multirow{4}{*}{$(\K^\times)^3/\Delta_{3,7,6}$}\cr
\multirow{4}{*}{$\begin{matrix} \tres_{\l,\m,\gamma, \delta}^{0,6} \vspace*{-0.15cm}
\end{matrix}$}  & 
    \resizebox{2.5cm}{.25cm}{
\xymatrix{ & \bullet^{2} \ar@/^.5pc/[dl] \ar@/^.4pc/[dr] & \cr
            \bullet^{1} \ar@/^.4pc/[ur]  \ar@/^.5pc/[rr] & &\bullet^{3} \ar@/^.4pc/[ll]\ar@/^.4pc/[ul]}}
&
\multirow{4}{*}{${\tiny\begin{pmatrix} 0 & \l &\m \\
1 & 0 &\gamma \\
\delta & 1 & 0\end{pmatrix}}$}
& \multirow{4}{*}{$(\K^\times,(\K^{\times})^4\cap D(\m-\l\delta\gamma))$}& \multirow{4}{*}{$\dfrac{(\K^\times)^4\cap D(\m-\l\delta\gamma)}{\Delta_{-2,3,6,7}}$}  \cr 
&&&&\\
  \hline
\end{tabular}
\caption{\footnotesize Simple three-dimensional algebras with $\dim(\D(A))=0$.}\label{table:2}
\end{table}

\subsubsection{$\boldsymbol{\dim(\D(A))=1}$}

Then $e(A)=4,5,6,7$. Without loss of generality, we can take a  natural basis for each type of evolution algebras so that the structure matrix relative to it is of the form shown in Table \ref{l1}.
In most cases, we could have chosen another natural basis with a different structure matrix (i.e. with a different distribution of $1$'s and parameters). The advantage of the selected bases is that the isomorphism conditions derived from this choice are valid for any field.

Observe that, when the graphs have the same number of edges,  as for instance in the cases $_1\tres_{\l,\mu}^{1,5}$, $_2\tres_{\l,\mu}^{1,5}$, $_3\tres_{\l,\mu}^{1,5}$ and $_4\tres_{\l,\mu}^{1,5}$, the associated evolution algebras are non-isomorphic since the graphs are non-isomorphic. 

For each of the type algebras appearing in Table \ref{l1}, except for $\tres^{1,7}_{\l,\m,\delta,\nu}$, the algebras are isomorphic if and only if the parameters coincide.

For the evolution algebras of type $\tres^{1,7}_{\l,\m,\delta,\nu}$, 
we obtain $\tres^{1,7}_{\l,\m,\delta,\nu}\cong \tres^{1,7}_{\l',\m',\delta',\nu'} $ if and only if $[(\l,\m,\delta,\nu)]=[(\l',\m',\delta',\nu')]$ in $((\K^{\times})^4\cap D(\mu+\l\delta\nu))/S_1$  with $S_1=\{M_1^k\,\colon\,k\in\Z\}$ (if $\ch({\K})=p$ with $p$ prime,  then $S_1=\Z_{2p}$ and if $\ch({\K})=0 $, then $S_1=\Z$) and $M_1=\footnotesize\begin{pmatrix}
0 & 1 &\phantom{-} 2 & \phantom{-}5\\
1 & 0 & \phantom{-}2 &\phantom{-} 4\\
0 & 0 &\phantom{-} 2 &\phantom{-} 3\\
0 & 0 & -1 & -2
\end{pmatrix}$ with the action defined in Definition \ref{action1}.

\subsubsection{$\boldsymbol{\dim(\D(A))=2}$ }

Then $e(A)=5,6,7,8$. Analogously to the previous case, we can consider a natural basis for each type of evolution algebras such that its structure matrix is of the form displayed in Table \ref{l2}. We could have chosen another natural basis with a different structure matrix (i.e. with a different distribution of $1$'s and parameters), for this choice the isomorphism conditions for the algebras presented in the table is valid for any field. 

As before, the evolution algebras $_i\tres_{\l,\mu,\delta}^{2,6}$ with $i\in\{1,2,3,4\}$  are non-isomorphic because their associated directed graphs are non-isomorphic. The same happens for $_i\tres_{\l,\mu,\delta,\nu}^{2,7}$ with $i\in\{1,2,3\}$.

Again, for each of the type algebras appearing in Table \ref{l2}, except for $_4\tres^{2,6}_{\l,\m,\delta}$ and $\tres^{2,8}_{\l,\m,\delta,\nu,\xi}$, the algebras are isomorphic if and only if the parameters coincide. 

For the evolution algebras of type $_4\tres^{2,6}_{\l,\m,\delta}$, we get 
$_4\tres^{2,6}_{\l,\m,\delta}\cong\, _4\tres^{2,6}_{\l',\m',\delta'} $ if and only if $[(\l,\m,\delta)]=[(\l',\m',\delta')]$ in $((\K^{\times})^3\cap D(\l\m+\delta))/S_2$ where $M_2=\footnotesize\begin{pmatrix}
-1 & 0 & 0 \\
\phantom{-}2 & 0 & 1 \\
\phantom{-}2 & 1 & 0
\end{pmatrix}$ and $S_2=\{M_2^k\,\colon\,k\in\Z\}=\{Id,M_2\}$ (isomorphic to $\Z_2)$ and with the action defined in Definition \ref{action1}. 


\begin{center}
\begin{table}[h]
\renewcommand{\arraystretch}{0.8}
\begin{tabular}{|c|C{2.3cm}|C{2cm}|c|}
\hline 
&&&\\
Type algebra & Graph & Structure matrix &  Orbit set \cr
&&&  \\
\hline
&&&\\[-0.2cm]
\multirow{4}{*} {$\tres_{\l}^{1,4}$} &  
\resizebox{2cm}{.2cm}{
\xymatrix{ & \bullet^{2}  \ar@/^.4pc/[dr] & \cr
            \bullet^{1}\ar@(dl,ul)\ar@/^.4pc/[ur]   & &\bullet^{3} \ar@/^.4pc/[ll]}}
  &
\multirow{4}{*}{$\tiny\begin{pmatrix}1 & 0 & \l\cr 
1 & 0 & 0\cr
0 & 1 & 0\end{pmatrix}$} & \multirow{4}{*}{\footnotesize{$\K^\times$}}\cr
&&&\\[-0.2cm]
&&&\\[-0.2cm]
\multirow{4}{*}{$_1\tres_{\l,\mu}^{1,5}$} &  
\centerline{
\resizebox{2cm}{.2cm}{
\xymatrix{ & \bullet^{2} \ar@/^.5pc/[dl] \ar@/^.4pc/[dr] & \cr
            \bullet^{1}\ar@(dl,ul)\ar@/^.4pc/[ur]   & &\bullet^{3} \ar@/^.4pc/[ll]}}}
&
\multirow{4}{*}{
$\tiny\begin{pmatrix}1 & \m & \l\cr 
1 & 0 & 0\cr
0 & 1 & 0\end{pmatrix}$}&  \multirow{4}{*}{\footnotesize{$(\K^\times)^2$}}\cr
&&&\\[-0.2cm]
&&&\\[-0.2cm]
\multirow{4}{*}{$_2\tres_{\l,\mu}^{1,5}$} &  
\resizebox{2cm}{.2cm}{
\xymatrix{ & \bullet^{2}  \ar@/^.4pc/[dr] & \cr
            \bullet^{1}\ar@(dl,ul)\ar@/^.4pc/[ur] \ar@/^.5pc/[rr]  & &\bullet^{3}  \ar@/^.4pc/[ll]}}
&
\multirow{4}{*}{$\tiny\begin{pmatrix}1 & 0 & \l\cr 
1 & 0 & 0\cr
\m & 1 & 0\end{pmatrix}$} &\multirow{4}{*}{\footnotesize{$(\K^\times)^2$}}\cr
&&&\\[-0.2cm]
&&&\\ \multirow{4}{*}{$\begin{matrix}
_3\tres_{\l,\m}^{1,5}
\end{matrix}$}
& 
\resizebox{2cm}{.2cm}{
\xymatrix{ & \bullet^{2}  \ar@/^.4pc/[dr] & \cr
            \bullet^{1}\ar@(dl,ul)\ar@/^.4pc/[ur]   & &\bullet^{3} \ar@/^.4pc/[ll] \ar@/^.4pc/[ul]}}
& \multirow{4}{*}{$\tiny\begin{pmatrix}1 & 0 & \l\cr 
1 & 0 & \mu\cr
0 & 1 & 0\end{pmatrix}$}& \multirow{4}{*}{\footnotesize{$(\K^\times)^2\cap D(\l-\m)$}}\cr
&&&\\[-0.2cm]
&&&\\[-0.2cm]
\multirow{4}{*}{$_4\tres_{\l,\m}^{1,5}$} &  
\resizebox{2cm}{.2cm}{
\xymatrix{ & \bullet^{2} \ar@/^.5pc/[dl] \ar@/^.4pc/[dr] & \cr
            \bullet^{1}\ar@(dl,ul)\ar@/^.4pc/[ur]   & &\bullet^{3} \ar@/^.4pc/[ul]}}
&
\multirow{4}{*}{$\tiny\begin{pmatrix}1 & \l & 0\cr 
1 & 0 & \mu\cr
0 & 1 & 0\end{pmatrix}$}  &\multirow{4}{*}{\footnotesize{$(\K^\times)^2$}}\cr
&&&\\[-0.2cm]
&&&\\[-0.2cm]
\multirow{4}{*}{$\begin{matrix}
_1\tres_{\l,\m,\delta}^{1,6}
\end{matrix}$} & 
\resizebox{2cm}{.2cm}{
\xymatrix{ & \bullet^{2} \ar@/^.5pc/[dl] \ar@/^.4pc/[dr] & \cr
            \bullet^{1}\ar@(dl,ul)\ar@/^.4pc/[ur]   & &\bullet^{3} \ar@/^.4pc/[ll] \ar@/^.4pc/[ul]}}
&
\multirow{4}{*}{$\tiny\begin{pmatrix}1 & \l & \mu\cr 
1 & 0 & \delta\cr
0 & 1 & 0\end{pmatrix}$}&\multirow{4}{*}{\footnotesize{$(\K^\times)^3\cap D(\mu-\delta)$}}\cr
&&&\\[-0.2cm]
&&&\\[-0.2cm]
\multirow{4}{*}{$\begin{matrix}
_2\tres_{\l,\m,\delta}^{1,6}
\end{matrix}$} & 
\resizebox{2cm}{.2cm}{
\xymatrix{ & \bullet^{2} \ar@/^.5pc/[dl] \ar@/^.4pc/[dr] & \cr
            \bullet^{1}\ar@(dl,ul)\ar@/^.4pc/[ur] \ar@/^.1pc/[rr]   & &\bullet^{3} \ar@/^.4pc/[ul]}}
&
\multirow{4}{*}{$\tiny\begin{pmatrix}1 & \l & 0\cr 
1 & 0 & \m\cr
\delta & 1 & 0\end{pmatrix}$}&  \multirow{4}{*}{\footnotesize{$(\K^\times)^3\cap D(\l\delta-1))$}}\cr
&&&\\[-0.2cm]
&&&\\[-0.2cm]
\multirow{4}{*}{$_3\tres_{\l,\m,\delta}^{1,6}$}&  
\resizebox{2cm}{.2cm}{
\xymatrix{ & \bullet^{2} \ar@/^.5pc/[dl]  & \cr
            \bullet^{1}\ar@(dl,ul)\ar@/^.4pc/[ur]  \ar@/^.5pc/[rr] & &\bullet^{3} \ar@/^.4pc/[ll] \ar@/^.1pc/[ul]}}
&
\multirow{4}{*}{$\tiny\begin{pmatrix}1 & \m & \delta\cr 
\l & 0 & 1\cr
1 & 0 & 0\end{pmatrix}$} &\multirow{4}{*}{\footnotesize{$(\K^\times)^3$}}\cr
&&&\\[-0.2cm]
&&&\\[-0.2cm]
\multirow{4}{*}{$\begin{matrix} \tres_{\l,\m,\delta,\nu}^{1,7}
\end{matrix}$} & 
    \resizebox{2cm}{.2cm}{
\xymatrix{ & \bullet^{2} \ar@/^.5pc/[dl] \ar@/^.4pc/[dr] & \cr
            \bullet^{1} \ar@(dl,ul) \ar@/^.4pc/[ur]  \ar@/^.5pc/[rr] & &\bullet^{3} \ar@/^.4pc/[ll]\ar@/^.4pc/[ul]}}
&
\multirow{4}{*}{$\tiny\begin{pmatrix}1 & \l & \m\cr 
1 & 0 & \delta\cr
\nu & 1 & 0\end{pmatrix}$}& \multirow{4}{*}{$\frac{(\K^\times)^4\cap D(\mu+\l\delta\nu)}{S_1}$}\cr
 & & & \\
\hline
\end{tabular}
\caption{\footnotesize Simple three-dimensional algebras with $\dim(\D(A))=1$ and $\l$, $\delta$, $\mu$ and $\nu$ nonzero. \\ 
\centering The acting group is $\KN$.}\label{l1}
\end{table}
\end{center}

For the evolution algebras of type $\tres_{\l,\m,\delta,\nu,\xi}^{2,8}$, 
we obtain that $\tres_{\l,\m,\delta,\nu,\xi}^{2,8}\cong \tres_{\l',\m',\delta',\nu',\xi'}^{2,8} $ if and only if $[(\l,\m,\delta,\nu,\xi)]=[(\l',\m',\delta',\nu',\xi')]$ in $((\K^{\times})^5\cap D(\xi(\delta\m-1)-\nu(\m-\l)))/S_3$, where $M_3=\footnotesize\begin{pmatrix}
0 & \phantom{-}0 & 1 & 0 & 0 \\
0 & -1  & 0 & 0 & 0\\
1 & \phantom{-}0 & 0 & 0 & 0\\
0 & \phantom{-}2 & 0 & 0 & 1\\
0 & \phantom{-}2 & 0 & 1 & 0
\end{pmatrix}$ and $S_3=\{M_3^k\,\colon\,k\in\Z\}=\{Id,M_3\}$ (isomorphic to $\Z_2$)  with the action defined in Definition \ref{action1}.

\subsubsection{$\boldsymbol{\dim(\D(A))=3}$}
Then $e(A)=6,7,8,9$. Again, we can choose a natural basis for each type of evolution algebras so that its structure matrix is of the form given in Table \ref{l3}. As before, we could have chosen another natural basis with a different structure matrix (i.e. with a different distribution of $1$'s and parameters), however the isomorphism condition for the algebras presented in the table is valid for any field. 
For evolution algebras of type $_1\tres^{3,7}_{\l,\m,\delta,\nu}$ and $\tres^{3,8}_{\l,\m,\delta,\nu,\xi}$, the algebras are isomorphic if and only if the parameters coincide.

For the evolution algebras of type $\tres_{\l,\m,\delta}^{3,6}$, we have  
$\tres_{\l,\m,\delta}^{3,6}\cong \tres_{\l',\m',\delta'}^{3,6} $ if and only if
$[(\l,\m,\delta)]=[(\l',\m',\delta')]$ in $((\K^\times)^3\cap D(\l\m+\delta))/S_4$, where 
$M_4=\footnotesize\begin{pmatrix}
0 & 1 & -2 \\
2 & 0 & \phantom{-}1 \\
1 & 0 & \phantom{-}0
\end{pmatrix}$ and $S_4=\{M_4^k\,\colon\,k\in\Z\}=\{Id,M_4,M_4^2\}$ (isomorphic to $\Z_3$)  with the action defined in Definition \ref{action1}. In this case, the Fröbenius normal form is 
$F_{4}=\footnotesize\begin{pmatrix}
0 & 1 & 0 \\
0 & 0 & 1 \\
1 & 0 & 0
\end{pmatrix}$ in any field. The moduli set is isomorphic to  $((\K^{\times})^3\cap D(\l\m-\delta^3))/\Z_3$ with the action $F_{4}\cdot(\l,\m,\delta)=(\m,\delta,\l)$.

For the evolution algebras of type $_2\tres_{\l,\m,\delta,\nu}^{3,7}$,
we obtain that $_2\tres_{\l,\m,\delta,\nu}^{3,7}\cong\,  _2\tres_{\l',\m',\delta',\nu'}^{3,7}$ if and only if
$[(\l,\m,\delta,\nu)]=[(\l',\m',\delta',\nu')]$ in $((\K^{\times})^4\cap D(\lambda\mu\nu-\nu +\xi))/S_5$, where  $M_5=\footnotesize\begin{pmatrix}
0 & 0 & 1 & -2 \\
0 & 0 & 0 & \phantom{-}1 \\
1 & 2 & 0 & \phantom{-}0\\
0 & 1 & 0 & \phantom{-}0
\end{pmatrix}$ and $S_5=\{M_5^k\,\colon\,k\in\Z\}=\{Id,M_5\}$ (isomorphic to $\Z_2$) with the action defined in Definition \ref{action1}.

For the evolution algebras of type $\tres_{\l,\m,\delta,\nu,\xi,\gamma}^{3,9}$, 
we get that $\tres_{\l,\m,\delta,\nu,\xi,\gamma}^{3,9}\cong \tres_{\l',\m',\delta',\nu',\xi',\gamma'}^{3,9} $  if and only if there is an element $X$ in the subgroup generated by $M_6$ and $M_7$ such that 
$X\cdot (\l,\m,\delta,\nu,\xi,\gamma)=(\l',\m',\delta',\nu',\xi',\gamma')$ 

\medskip
\noindent
with  $M_6=\footnotesize\begin{pmatrix}
0 & \phantom{-}0 & 1 & 0 & 0 & 0 \\
0 & -1 & 0 & 0 & 0 & 0 \\
1 & \phantom{-}0 & 0 & 0 & 0 & 0\\
0 & \phantom{-}2 & 0 & 0 & 1 & 0\\
0 & \phantom{-}2 & 0 & 1 & 0 & 0\\
0 & \phantom{-}1 & 0 & 0 & 0 & 1
\end{pmatrix}$
and 
$M_{7}=\footnotesize\begin{pmatrix}
\phantom{-}0 & 0 & 0 & 1 & 0 & -2 \\
-1 & 0 & 0 & 0 & 1 & -2 \\
\phantom{-}0 & 1 & 0 & 0 & 0 & \phantom{-}1\\
\phantom{-}2 & 0 & 0 & 0 & 0 & \phantom{-}1\\
\phantom{-}2 & 0 & 1 & 0 & 0 & \phantom{-}0\\
\phantom{-}1 & 0 & 0 & 0 & 0 & \phantom{-}0
\end{pmatrix}.$

The moduli set is isomorphic to ${((\K^\times)^6\cap D(\xi(\delta\mu-1)+\nu(\l-\mu)- \gamma(\delta\l-1)))}/{\mathbb{S}_3}$, where $\mathbb{S}_3$ is the permutation group of three elements. In fact, this group is generated by the element $M_6$, which is of order $2$, and $M_7$, which is of order $3$. Moreover, $M_6M_7=M_7^2M_6$. This is the only case where the acting group is nonabelian.

\medskip

In conclusion, we can summarize the classification of three-dimensional simple evolution algebras over arbitrary fields in the below theorem.

\begin{theorem}\label{teorema2}
Let $A$ be a three-dimensional simple evolution $\K$-algebra. We distinguish several cases:
\begin{enumerate}[\rm (i)]
    \item If $\dim(\D(A))=0$, then $A$ is isomorphic to one and only one of the type evolution algebras $\tres_{\l}^{0,3}$, $\tres_{\l,\m}^{0,4}$, $\tres_{\l,\m,\gamma}^{0,5}$ and $\tres_{\l,\m,\gamma,\delta}^{0,6}$.
    \item If $\dim(\D(A))=1$, then $A$ is isomorphic to one and only one of the type evolution algebras $\tres_{\l}^{1,3}$, $_i\tres_{\l,\m}^{1,5}$ with $i\in \{1,2,3,4\}$,  $_j\tres_{\l,\m,\delta}^{1,6}$ for $j\in \{1,2,3\}$ and $\tres_{\l,\m,\delta, \nu}^{1,7}$.
    \item  If $\dim(\D(A))=2$, then $A$ is isomorphic to one and only one of the type evolution algebras $\tres_{\l,\mu}^{2,5}$, $_i\tres_{\l,\m,\delta}^{2,6}$ with $i\in \{1,2,3,4\}$, $_j\tres_{\l,\m,\delta,\nu}^{2,7}$ for $j\in \{1,2,3\}$ and $\tres_{\l,\m,\delta,\nu, \xi}^{2,8}$.
    
     \item  If $\dim(\D(A))=3$, then $A$ is isomorphic to one and only one of the type evolution algebras $\tres_{\l,\m,\delta}^{3,6}$, $_i\tres_{\l,\m,\delta,\nu}^{3,7}$ with $i\in\{1,2\}$,  $\tres_{\l,\m,\delta,\nu,\xi}^{3,8}$ and $\tres_{\l,\m,\delta,\nu,\xi,\gamma}^{3,9}$.

    \end{enumerate}
\end{theorem}

 \begin{center}
\begin{table}[H]
\renewcommand{\arraystretch}{0.8}
\begin{tabular}{|c|C{2cm}|C{2.5cm}|c|}
\hline 
&&& \\
Type algebra  & Graph & Structure matrix &  Orbit set \cr
&&&  \\
\hline
&&& \\[-0.2cm]
&&& \\[-0.2cm]
&&& \\[-0.2cm]
&&&\\[-0.2cm]
\multirow{4}{*} {$\tres_{\l,\m}^{2,5}$} &  
\resizebox{1.7cm}{.2cm}{
\xymatrix{ & \bullet^{2} \ar@(ur,ul) \ar@/^.4pc/[dr] & \cr
            \bullet^{1}\ar@(dl,ul)\ar@/^.4pc/[ur]   & &\bullet^{3} \ar@/^.4pc/[ll]}}
  &
\multirow{4}{*}{$\tiny\begin{pmatrix}1 & 0 & \mu\cr 
\lambda & 1 & 0\cr
0 & 1 & 0\end{pmatrix}$}
&
\multirow{4}{*}{\footnotesize{$(\K^\times)^2$}}\cr
&&&\\[-0.2cm]
&&&\\[-0.2cm]
\hline
&&&\\[-0.2cm]
&&& \\[-0.2cm]
&&& \\[-0.2cm]
&&&\\[-0.2cm]
\multirow{4}{*}{$\begin{matrix}
_1\tres_{\l,\mu,\delta}^{2,6}
\end{matrix}$} & 
\resizebox{1.7cm}{.2cm}{
\xymatrix{ & \bullet^{2} \ar@(ur,ul) \ar@/^.4pc/[dr] & \cr
            \bullet^{1}\ar@(dl,ul)\ar@/^.4pc/[ur] \ar@/^.5pc/[rr]  & &\bullet^{3}  \ar@/^.4pc/[ll]}}
&
\multirow{4}{*}{
$\tiny\begin{pmatrix}1 & 0 & \delta\cr 
\lambda & 1 & 0\cr
\mu & 1 & 0\end{pmatrix}$}
&
\multirow{4}{*}{\footnotesize{$(\K^\times)^3\cap D(\l-\m)$}}\cr
&&&\\[-0.2cm]
&&&\\[-0.2cm]
\hline
&&&\\[-0.2cm]
&&& \\[-0.2cm]
&&& \\[-0.2cm]
&&&\\[-0.2cm]
\multirow{4}{*}{$_2\tres_{\l,\mu,\delta}^{2,6}$} &  
\centerline{
\resizebox{1.7cm}{.2cm}{
\xymatrix{ & \bullet^{2} \ar@(ur,ul)  \ar@/^.5pc/[dl] \ar@/^.4pc/[dr] & \cr
            \bullet^{1}\ar@(dl,ul)\ar@/^.4pc/[ur]   & &\bullet^{3} \ar@/^.4pc/[ll]}}} 
&
\multirow{4}{*}{$\tiny\begin{pmatrix}1 & \mu & \delta\cr 
\l & 1 & 0\cr
0 & 1 & 0\end{pmatrix}$}
&
\multirow{4}{*}{\footnotesize{$(\K^\times)^3$}}\cr
&&&\\[-0.2cm]
&&&\\[-0.2cm]
\hline
&&&\\[-0.2cm]
&&& \\[-0.2cm]
&&& \\[-0.2cm]
&&&\\[-0.2cm]
\multirow{4}{*}{
$_3\tres_{\l,\m,\delta}^{2,6}$}
& 
\resizebox{1.7cm}{.2cm}{
\xymatrix{ & \bullet^{2} \ar@(ur,ul) \ar@/^.5pc/[dl] \ar@/^.4pc/[dr] & \cr
            \bullet^{1}\ar@(dl,ul)\ar@/^.4pc/[ur]   & &\bullet^{3} \ar@/^.4pc/[ul]}}
& \multirow{4}{*}{$\tiny\begin{pmatrix}1 & \mu & 0\cr 
\l & 1 & \delta\cr
0 & 1 & 0\end{pmatrix}$}
&
\multirow{4}{*}{\footnotesize{$(\K^\times)^3$}}\cr
&&&\\[-0.2cm]
&&&\\[-0.2cm]
\hline
&&&\\[-0.2cm]
&&& \\[-0.2cm]
&&& \\[-0.2cm]
&&&\\[-0.2cm]
\multirow{4}{*}{$\begin{matrix}
_4\tres_{\l,\m,\delta}^{2,6}
\end{matrix}$}  &  
    \resizebox{1.7cm}{.2cm}{
\xymatrix{ & \bullet^{2}\ar@(ur,ul) \ar@/^.4pc/[dr] & \cr
            \bullet^{1} \ar@(dl,ul)   \ar@/^.5pc/[rr] & &\bullet^{3} \ar@/^.4pc/[ll]\ar@/^.4pc/[ul]}}
&
\multirow{4}{*}{$\tiny\begin{pmatrix}1 & 0 & \mu\cr 
0 & 1 & \delta\cr
\l & 1 & 0\end{pmatrix}$} 
&
\multirow{4}{*}{$\frac{(\K^\times)^3\cap D(\l\m+\delta)}{\Z_2}$}\cr
&&&\\[-0.2cm]
&&&\\[-0.2cm]
\hline
&&&\\[-0.2cm]
&& &\\[-0.2cm]
&& &\\[-0.2cm]
&&&\\[-0.2cm]
\multirow{5}{*}{$\begin{matrix}
_1\tres_{\l,\m,\delta,\nu}^{2,7}
\end{matrix}$} & 
\resizebox{1.7cm}{.2cm}{
\xymatrix{ & \bullet^{2} \ar@(ur,ul) \ar@/^.5pc/[dl] \ar@/^.4pc/[dr] & \cr
            \bullet^{1}\ar@(dl,ul)\ar@/^.4pc/[ur]   & &\bullet^{3} \ar@/^.4pc/[ll] \ar@/^.4pc/[ul]}}
&
\multirow{5}{*}{$\tiny\begin{pmatrix}1 & \mu & \delta\cr 
\l & 1 & \nu\cr
0 & 1 & 0\end{pmatrix}$}
&
\multirow{5}{*}{\footnotesize{$(\K^\times)^4\cap D(\l\delta-\nu)$}}\cr
&&&\\[-0.2cm]
&&&\\[-0.2cm]
\hline
&&&\\[-0.2cm]
&& &\\[-0.2cm]
&& &\\[-0.2cm]
&&&\\[-0.2cm]
\multirow{4}{*}{$\begin{matrix}
_2\tres_{\l,\m,\delta,\nu}^{2,7}
\end{matrix}$} & 
\resizebox{1.7cm}{.2cm}{
\xymatrix{ & \bullet^{2} \ar@(ur,ul) \ar@/^.5pc/[dl] \ar@/^.4pc/[dr] & \cr
            \bullet^{1}\ar@(dl,ul)\ar@/^.4pc/[ur] \ar@/^.1pc/[rr]   & &\bullet^{3} \ar@/^.4pc/[ul]}}
&
\multirow{4}{*}{$\tiny\begin{pmatrix}1 & \delta & 0\cr 
\l & 1 & \nu\cr
\m & 1 & 0\end{pmatrix}$}
&
\multirow{4}{*}{\footnotesize{$(\K^\times)^4\cap D(\m\delta\nu-\nu)$}}\cr
&&&\\[-0.2cm]
&&&\\[-0.2cm]
\hline
&&&\\[-0.2cm]
&& &\\[-0.2cm]
&&&\\[-0.2cm]
&&&\\[-0.2cm]
\multirow{4}{*}{$\begin{matrix}
_3\tres_{\l,\m,\delta,\nu}^{2,7}
\end{matrix}$}& 
    \resizebox{1.7cm}{.2cm}{
\xymatrix{ & \bullet^{2}\ar@(ur,ul)  \ar@/^.4pc/[dr] & \cr
            \bullet^{1} \ar@(dl,ul) \ar@/^.4pc/[ur]  \ar@/^.5pc/[rr] & &\bullet^{3} \ar@/^.4pc/[ll]\ar@/^.4pc/[ul]}}
&
\multirow{4}{*}{$\tiny\begin{pmatrix}1 & 0 & \delta\cr 
\l & 1 & \nu\cr
\m & 1 & 0\end{pmatrix}$}
&
\multirow{4}{*}{\footnotesize{$(\K^\times)^4\cap D(\delta(\l-\m)-\nu)$}}\cr
&&&\\[-0.2cm]
&&&\\[-0.2cm]
\hline
&&&\\[-0.2cm]
&&& \\[-0.2cm]
&&& \\[-0.2cm]
&&&\\[-0.2cm]
\multirow{4}{*}{$\begin{matrix}
\tres_{\l,\m,\delta,\nu,\xi}^{2,8}
\end{matrix}$} &  
    \resizebox{1.7cm}{.2cm}{
\xymatrix{ & \bullet^{2}\ar@(ur,ul) \ar@/^.5pc/[dl] \ar@/^.4pc/[dr] & \cr
            \bullet^{1} \ar@(dl,ul) \ar@/^.4pc/[ur]  \ar@/^.5pc/[rr] & &\bullet^{3} \ar@/^.4pc/[ll]\ar@/^.4pc/[ul]}}
&
\multirow{4}{*}{$\tiny\begin{pmatrix}1 & \delta & \nu\cr 
\l & 1 & \xi\cr
\mu & 1 & 0\end{pmatrix}$}&
\multirow{4}{*}{$\frac{(\K^\times)^5\cap D(\xi(\delta\m-1)-\nu(\m-\l))}{\Z_2}$}\cr
 & &  &\\
\hline
\end{tabular}
\caption{\footnotesize Simple three-dimensional algebras with $\dim(\D(A))=2$ and $\l$, $\delta$, $\mu$,$\nu$ and $\xi$ nonzero. \\
\centering The acting  group is $\KN$.}\label{l2}
\end{table}
\end{center}

 \begin{center}
\begin{table}[h]
\renewcommand{\arraystretch}{0.8}
\begin{tabular}{|c|C{2cm}|C{2cm}|c|}
\hline 
&&& \\
Type algebra  & Graph & Structure matrix &  Orbit set \cr
&&&  \\
\hline
&&& \\[-0.2cm]
&&& \\[-0.2cm]
&&& \\[-0.2cm]
&&&\\[-0.2cm]
\multirow{4}{*} {$\begin{matrix}
\tres_{\l,\m,\delta}^{3,6}
\end{matrix}$} &  
\resizebox{1.7cm}{.2cm}{
\xymatrix{ & \bullet^{2} \ar@(ur,ul) \ar@/^.4pc/[dr] & \cr
            \bullet^{1}\ar@(dl,ul)\ar@/^.4pc/[ur]   & &\bullet^{3}\ar@(dr,ur) \ar@/^.4pc/[ll]}}
  &
\multirow{4}{*}{$\tiny\begin{pmatrix}1 & 0 & \mu\cr 
\lambda & 1 & 0\cr
0 & 1 & \delta\end{pmatrix}$}
&
\multirow{4}{*}{$\frac{(\K^\times)^3\cap D(\l\m+\delta)}{\Z_3}$}\cr
&&&\\[-0.2cm]
&&&\\[-0.2cm]
\hline
&&&\\[-0.2cm]
&&& \\[-0.2cm]
&&& \\[-0.2cm]
&&&\\[-0.2cm]
\multirow{4}{*}{$\begin{matrix}
_1\tres_{\l,\mu,\delta,\nu}^{3,7}
\end{matrix}$} & 
\resizebox{1.7cm}{.2cm}{
\xymatrix{ & \bullet^{2} \ar@(ur,ul) \ar@/^.4pc/[dr] & \cr
            \bullet^{1}\ar@(dl,ul)\ar@/^.4pc/[ur] \ar@/^.5pc/[rr]  & &\bullet^{3}\ar@(dr,ur)  \ar@/^.4pc/[ll]}}
&
\multirow{4}{*}{
$\tiny\begin{pmatrix}1 & 0 & \delta\cr 
\lambda & 1 & 0\cr
\mu & 1 & \nu\end{pmatrix}$}
&
\multirow{4}{*}{\footnotesize{$(\K^\times)^4\cap D(\m-\l-1)$}}\cr
&&&\\[-0.2cm]
&&&\\[-0.2cm]
\hline
&&&\\[-0.2cm]
&&& \\[-0.2cm]
&&& \\[-0.2cm]
&&&\\[-0.2cm]
\multirow{4}{*}{
$\begin{matrix}
_2\tres_{\l,\m,\delta,\nu}^{3,7}
\end{matrix}$}
& 
\resizebox{1.7cm}{.2cm}{
\xymatrix{ & \bullet^{2} \ar@(ur,ul) \ar@/^.5pc/[dl] \ar@/^.4pc/[dr] & \cr
            \bullet^{1}\ar@(dl,ul)\ar@/^.4pc/[ur]   & &\bullet^{3}\ar@(dr,ur) \ar@/^.4pc/[ul]}}
& \multirow{4}{*}{$\tiny\begin{pmatrix}1 & \mu & 0\cr 
\l & 1 & \delta\cr
0 & 1 & \nu\end{pmatrix}$}
&
\multirow{4}{*}{$\frac{(\K^{\times})^4\cap D(\l\m\nu-\nu+\delta)}{\Z_2}$}\cr
&&&\\[-0.2cm]
&&&\\[-0.2cm]
\hline
&&&\\[-0.2cm]
&& &\\[-0.2cm]
&& &\\[-0.2cm]
&&&\\[-0.2cm]
\multirow{4}{*}{$\begin{matrix}
\tres_{\l,\m,\delta,\nu,\xi}^{3,8}
\end{matrix}$} & 
\resizebox{1.7cm}{.2cm}{
\xymatrix{ & \bullet^{2} \ar@(ur,ul) \ar@/^.5pc/[dl] \ar@/^.4pc/[dr] & \cr
            \bullet^{1}\ar@(dl,ul)\ar@/^.4pc/[ur] \ar@/^.1pc/[rr]   & &\bullet^{3}\ar@(dr,ur) \ar@/^.4pc/[ul]}}
&
\multirow{4}{*}{$\tiny\begin{pmatrix}1 & \delta & 0\cr 
\l & 1 & \xi\cr
\m & 1 & \nu\end{pmatrix}$}
&
\multirow{4}{*}{\footnotesize{$(\K^\times)^5\cap D(\xi(\delta\mu-1)-\nu(\delta\l-1))$}}\cr
&&&\\[-0.2cm]
&&&\\[-0.2cm]
\hline
&&&\\[-0.2cm]
&&& \\[-0.2cm]
&&& \\[-0.2cm]
&&&\\[-0.2cm]
\multirow{4}{*}{$\begin{matrix}
\tres_{\l,\m,\delta,\nu,\xi,\gamma}^{3,9}
\end{matrix}$} 
&  
    \resizebox{1.7cm}{.2cm}{
\xymatrix{ & \bullet^{2}\ar@(ur,ul) \ar@/^.5pc/[dl] \ar@/^.4pc/[dr] & \cr
            \bullet^{1} \ar@(dl,ul) \ar@/^.4pc/[ur]  \ar@/^.5pc/[rr] & &\bullet^{3}\ar@(dr,ur) \ar@/^.4pc/[ll]\ar@/^.4pc/[ul]}}
&
\multirow{4}{*}{$\tiny\begin{pmatrix}1 & \delta & \nu\cr 
\l & 1 & \xi\cr
\mu & 1 & \gamma\end{pmatrix}$}&
\multirow{4}{*}{$\frac{(\K^\times)^6\cap D(\xi(\delta\mu-1)+\nu(\l-\mu)- \gamma(\delta\l-1))}{\mathbb{S}_3}$}\cr
&& &  \\
\hline
\end{tabular}
\caption{\footnotesize Simple three-dimensional algebras with $\dim(\D(A))=3$ and $\l$, $\delta$, $\mu$, $\nu$, $\xi$ and $\gamma$ nonzero. \\ 
\centering The  acting group is $\KN$.}\label{l3}
\end{table}
\end{center}

\vspace*{-1cm}

\section{Simple and semisimple evolution algebras arising as the tensor product of two simple ones}

Let $E=(E^0,E^1,r_E,s_E)$ and $F=(F^0,E^1,r_F,s_F)$ be two directed graphs. We recall that the \textit{categorical product} of $E$ and $F$ is the directed graph defined by $E \times F := (E^0 \times F^0, E^1 \times F^1, r, s)$ where $s(f,g) =(s(f), s(g))$ and $r(f,g)=(r(f), r(g))$ for any $(f,g)\in E^1 \times F^1$.

We know, by \cite{CMMT}, that  the tensor product of two evolution algebras is an evolution algebra as well. Moreover, if $E=(E^0,E^1,r_E,s_E)$ and $F=(F^0,
F^1,r_F,s_F)$  are the directed graphs associated to the evolution algebras $A_1$ and $A_2$ respectively, then the directed graph associated to $A_1 \otimes A_2$ is the graph $E \times F$.

Following \cite{McANDREW}, let $E=(E^0,E^1,r_E,s_E)$ be a directed graph and $u, v\in E^0$. We say $u, v$ are \textit{strongly connected} if there exists a path from $u$ to $v$ and from $v$ to $u$ or $u=v$. This is an equivalence relation and we define a \textit{component} $C$ of $E$ to be a subgraph whose vertices, $C^0$, are the vertices in an equivalence class with respect to this relation and whose edges are those  $f \in E^1$ with $s(f), r(f) \in C^0$.

One can find, in \cite{McANDREW}, the following result that concerns strongly connected directed graphs (involving the categorical product of graphs).

\begin{theorem}[\cite{McANDREW}, Theorem 1.(ii)]\label{thm:strong}
Let $G_1$ and $G_2$ be strongly connected directed graphs. Let $$d_1=d(G_1)=gcd\{\text{\rm length of all the closed paths in } G_1\},$$
$$d_2=d(G_2)=gcd\{\text{\rm length of all the closed paths in } G_2\}$$ and
$$d_3=gcd\{d_1,d_2\}.$$
Then the number of components of $G_1\times G_2$ is $d_3$.
\end{theorem}

\begin{remark}\label{remark:strongly}\rm
The previous theorem implies that if $E$ and $F$ are strongly connected graphs, then the connected components of $E\times F$ are strongly connected.
\end{remark}

So, we have a result that improves \cite[item (i) Corollary 4.3]{CMMT}.

\begin{theorem}\label{cor:strong}
Let $A_i$ ($i=1,2$) be finite-dimensional simple evolution algebras over a field $\K$ whose associated directed graphs are $E$ and $F$, respectively. We have:
\begin{enumerate}[\rm (i)]
    \item The evolution algebra $A_1\otimes A_2$ is simple if and only if $1=\gcd(d_1,d_2)$ where 
$d_1$ and $d_2$ are as in Theorem \ref{thm:strong}. In particular, if $E$ and $F$ have closed paths of coprime length, then $A_1\otimes A_2$ is simple. 
\item The evolution algebra $A_1\otimes A_2$ is semisimple (a direct sum of simple evolution algebras).
\end{enumerate}

\end{theorem}
\begin{proof}
For item (i), if $A_1\otimes A_2$ is simple, then the graph $E\times F$ has one component hence $\gcd(d_1,d_2)=d_3=1$. Conversely, if $1=d_3=\gcd(d_1,d_2)$ then $E\times F$ has one  component which is $(E\times F)^0$. Recall that any component is strongly connected, moreover, the algebra $A_1\otimes A_2$ is perfect since both $A_1$ and $A_2$ are perfect (being simple) by \cite[item (i) Proposition 2.14]{CMMT}. So, the graph $E\times F$ is strongly connected and $A_1\otimes A_2$ is perfect. Consequently, $A_1\otimes A_2$ is simple by \cite[Proposition 2.7]{CKS1}.
The second assertion is straightforward. For item (ii), if $d_3\geq 1$ each component gives rise to a simple evolution algebra.
\end{proof}

\begin{corollary}\label{patatas}
Let $A_1$ and $A_2$ be two finite-dimensional simple evolution algebras over a field $\K$. Suppose 
$\dim(\D({A_i})) > 0$ for some $i \in \{1,2\}$, then $A_1 \otimes A_2$ is a simple evolution algebra.
\end{corollary}

\begin{proof}
Since $\dim(\D({A_i})) > 0$, for some $i \in \{1,2\}$, we get that $A_1$ or $A_2$ has at least one loop.   This means that $d_1$ or $d_2$ is $1$. The result is straightforward.   
\end{proof}

The next issue is how to construct simple tensorially decomposable evolution algebras. 

\begin{remark}\label{nuevas}\rm

Note that, item (i) Theorem \ref{cor:strong} and  Corollary \ref{patatas} allow us to construct more complex models of simple evolution algebras. For instance, taking as model the
simple algebra $\tres^{1,4}_\l$, whose structure matrix is
$$\tiny\begin{pmatrix}1 & 0 &\l\\ 1 & 0 & 0\\ 0 & 1 & 0\end{pmatrix},$$ and replacing each $1$ in the matrix with an $n\times n$ structure matrix $M$ of another simple evolution algebra (and replacing each $\l$ with $\l M$), we  get the structure matrix $\tiny\begin{pmatrix}M & 0 &\l M\\ M & 0 & 0\\ 0 & M & 0\end{pmatrix}$, which represents a $3n$-dimensional simple evolution algebras. So, we have many models of $3n$ and $2n$-dimensional simple algebras from our classification theorems. 

Although, observe that not every simple tensorially decomposable evolution algebra requires 
the factors to have all closed paths of coprime length. We illustrate this fact.
Let $A_1$ and $A_2$ be two simple finite dimensional  evolution algebras over a field $\K$ with associated directed graphs $E$ and $F$ as follows:

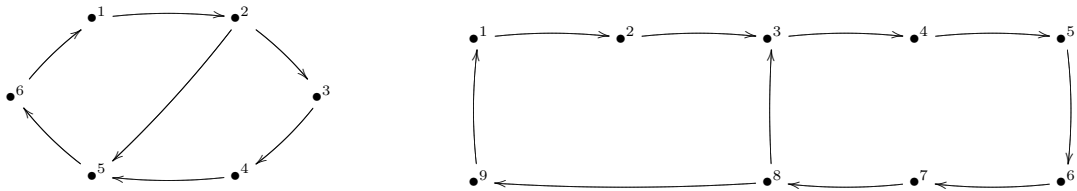
\begin{figure}[h]
\centering
      \begin{subfigure}[h]{0.45\textwidth}
      \centering
   $\resizebox{.6\textwidth}{!}{\xymatrix{ & \bullet^{1}\ar@/^.2pc/[rr] &  & \bullet^{2}\ar@/^.2pc/[dr]  \ar@/^.2pc/[ddll] &\cr
           \bullet^{6} \ar@/^.2pc/[ur]  & &  & &\bullet^{3} \ar@/^.2pc/[dl]\cr
          & \bullet^{5} \ar@/^.2pc/[ul] &  &  \bullet^{4} \ar@/^.2pc/[ll] & } } $
       \caption{\footnotesize Directed graph $E$ associated to the evolution algebra $A_1$.}
          \end{subfigure}
 \begin{subfigure}[h]{0.5\textwidth}
 \centering
        $ \resizebox{0.99\textwidth}{!}{\xymatrix{\bullet^{1} \ar@/^.2pc/[rr] & & \bullet^{2}\ar@/^.2pc/[rr] & & \bullet^{3}\ar@/^.2pc/[rr] & & \bullet^{4}\ar@/^.2pc/[rr] &  & \bullet^{5}\ar@/^.2pc/[dd]  \cr
&&&&&&&& \cr
          \bullet^{9} \ar@/^.2pc/[uu]& & & &  \bullet^{8} \ar@/^.1pc/[uu]\ar@/^.2pc/[llll]  & & \bullet^{7}\ar@/^.2pc/[ll]  & & \bullet^{6} \ar@/^.2pc/[ll] \cr}}$
              \caption{\footnotesize Directed graph $F$ associated to the evolution algebra $A_2$.}
         \end{subfigure}
\end{figure}

In the same vein of Theorem \ref{thm:strong}, $d_1=gcd\{4,6\}=2$ and $d_2=gcd\{6,9\}=3$. Hence $d_3=gcd\{d_1,d_2\}=1$. So, $E\times F$ has one strongly connected component. Hence, $A_1\otimes A_2$ is simple.

 $$\includegraphics[width=8cm]{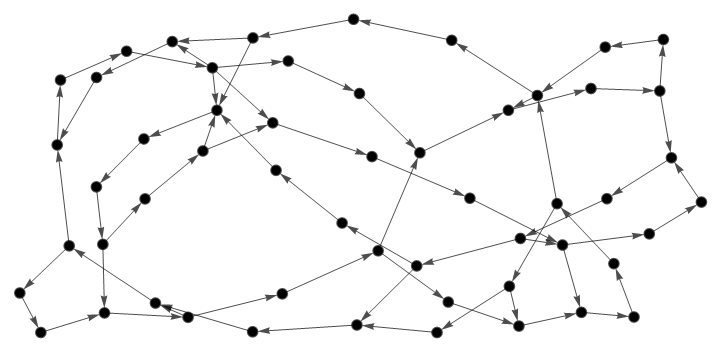}$$
 \begin{center}
    \footnotesize {\rm (3)} Directed graph $E\times F$ associated to the evolution algebra $A_1\otimes A_2.$
 \end{center}

\end{remark}
Recall the definitions of tensorially decomposable and tensorially indecomposable $\K$-algebras.
\begin{definition}[\cite{CMMT}, Definition 2.9] \rm
 We say that a $\K$-algebra $A$ is \textit{tensorially decomposable} if it is isomorphic to $A_1\otimes A_2$ where $A_1$ and $A_2$ are $\K$-algebras with $\dim(A_1),\dim(A_2)> 1$. Otherwise we say that $A$ is \textit{tensorially indecomposable}.
\end{definition}

One of the motivations of the following theorem is to  decide if we could construct simple evolution algebras arising as quotients of tensor products of other evolution algebras. Here we prove that if the ideal realizing the quotient is $1$-dimensional, then the quotient algebra is not simple.

\begin{theorem}\label{cociente}
Let $A$ be a finite dimensional evolution algebra over a field $\K$ with natural basis $B$.
\begin{enumerate}[\rm (i)]
    \item     If $I\triangleleft A$ is an  ideal with $I=\K u$ and $\Supp_B{(u)}> 1$, then $A/I$ is not simple. 
    \item If $A= A_1 \otimes A_2$ is  tensorially decomposable and $0\ne I\triangleleft A$ is an  ideal with $I=\K u$, then $A/ I$ is not simple.
   \end{enumerate}
    
\end{theorem}

\begin{proof}
For (i) write $u=\sum_h \l_h e_h$ where $B=\{e_h\}_{h=1}^n$ is a natural basis of $A$. We may assume that $\l_1,\l_2\ne 0$.
Then $I\ni ue_i$ for $i=1,2$ so
$e_1^2,\ e_2^2\in I$ thus 
$A^2\subset I+\sum_{h=3}^n \K e_h^2$
and $\dim(A^2)\le n-1$. Since $\dim[(A^2+I)/I]=\dim(A^2/A^2\cap I)\le n-2$ we have a nonzero proper ideal
$(A^2+I)/I\triangleleft A/I$.

For (ii), suppose that $B_1=\{a_i\}_{i=1}^n$, $B_2=\{b_j\}_{j=1}^q$ are natural bases of $A_1$ and $A_2$ respectively, and $B_1\otimes B_2=\{e_h\}_{h=1}^{n q}$ a natural basis of $A=A_1\otimes A_2$. So, $u=\sum_{h=1}^n \lambda_h e_h$. %
If $\Supp_B{(u)}>1$ we apply item (i) and if  $\Supp_B{(u)}=1$ then $u=\lambda a_1\otimes b_1$, $\lambda\neq 0$ (so we may assume without loss of generality $\l=1$).
Since $I=\K u$, we have either $u^2=0$ or $u^2\in \K^{\times} u$. In the last case, we can take $u$ being idempotent (re-scaling if necessary). 
\begin{enumerate}
\item In the case $u^2=u$ we have $ a_1\otimes b_1=u= u^2=a_1^2\otimes b_1^2=\sum_{i=1}^n\omega_{i1}a_i\otimes \sum_{j=1}^q\sigma_{j1}b_j=\sum_{i,j}\omega_{i1}\sigma_{j1}a_i\otimes b_j$.
So $\omega_{11}\sigma_{11}\ne 0$ and
$\omega_{i1}\sigma_{j1}=0$ if $i\neq 1$ and $j\neq 1$. Now, if we suppose $\omega_{i1}\neq 0$ for $i\neq 1$, then $\sigma_{j1}=0$ for all $j\neq 1$. Hence $b_1^2=\sigma_{11}b_1$. 
Thus $A_2$ has a $1$-dimensional ideal $J=\K b_1$ and therefore $A$ has an ideal $A_1\otimes J$ with $\dim(A_1\otimes J)=\dim(A_1)$. 
Then $A/I$ has the ideal $(A_1\otimes J+I)/I$ which is nonzero and proper. 
\item If $u^2=0$ we deduce, as before, $\omega_{i1}\sigma_{j1}=0$
for any $i$ and $j$. Then, for instance if $\omega_{i1}\ne 0$, we have
$\sigma_{j1}=0$ for every $j$, so $A_2$ has a $1$-dimensional ideal and we conclude as above.
\end{enumerate}
\end{proof}

\section*{Acknowledgments}
 The  four   authors are supported by the Junta de Andaluc\'{\i}a  through projects UMA18-FEDERJA-119  and FQM-336 and  by the Spanish Ministerio de Ciencia e Innovaci\'on   through project  PID2019-104236GB-I00,  all of them with FEDER funds.

{\bf Data Availability Statement:}  The authors confirm that the data supporting the findings of this study are available within the article.


\begin{thebibliography}{99}


\bibitem{boudi20} 
Nadia Boudi, Yolanda Cabrera Casado and Mercedes Siles Molina. {Natural families in evolution algebras.} \textit{Publicacions Matemàtiques.}   \textbf{66} (1) (2022).

\bibitem{YC} Yolanda Cabrera Casado.  {Evolution algebras}. \textit{ Doctoral dissertation. Universidad de M\'alaga (2016)}. http://hdl.handle.net/10630/14175.

\bibitem{CCGMM} Yolanda Cabrera Casado, Maria Inez Cardoso Gon\c calves, Daniel Gon\c calvez, Dolores Mart\'in Barquero and C\'andido Mart\'in Gonz\'alez. {Chains in evolution algebras.}  \textit{Linear Algebra Appl.} \textbf{622} (2021) 104-149.

\bibitem{CKS1} Yolanda Cabrera Casado, M\"uge Kanuni and Mercedes Siles Molina.  {Basic ideals in evolution algebras}. \textit{Linear Algebra Appl.} \textbf{570} (2019)  148-180.

\bibitem{CMMT} Yolanda Cabrera Casado, Dolores Martín Barquero, Cándido Martín González and Alicia Tocino. {Tensor product of evolution algebras}. In Press. \textit{Mediterranean Journal of Mathematics}. ArXiv:2111.06114  (2021).

\bibitem{YE} Yolanda Cabrera Casado and Elkin Quintero Vanegas. {Absorption Radical of an evolution algebra}. {Preprint.}


\bibitem{CSV2} Yolanda Cabrera Casado, Mercedes Siles Molina and M. Victoria Velasco. Classification of three dimensional evolution algebras.  \textit{Linear Algebra Appl.} \textbf{524} (2017) 68-108.


\bibitem{ceballos} Manuel Ceballos González, Raúl M. Falcon Ganfornina, Juan Nuñez Valdes and  Ángel F. Tenorio Villalón. {A historical perspective of Tian's evolution algebras}.  In Press. \textit{Expositiones Mathematicae.} (2021).
\url{https://doi.org/10.1016/j.exmath.2021.11.004}

\bibitem{BMV} M. Eugenia Celorrio and M. Victoria Velasco. {Classifying Evolution Algebras of Dimensions Two and Three.} \textit{Mathematics.}  \textbf{7} (12) (2019).



\bibitem{EL1} Alberto Elduque and Alicia Labra. {Evolution algebras and graphs}. \textit{J. Algebra Appl.}  \textbf{14} (7) (2015), 1550103, 10 pp.

\bibitem{IM} A.N. Imomkulov. {Classification of a family of three dimensional real evolution algebras.} \textit{TWMS J. Pure Appl. Math.} \textbf{10} (2) (2019) 225-238.



\bibitem{McANDREW} M. H. Mc Andrew. On the product of directed graphs. \textit{Proceedings of the American Mathematical Society}
Vol. 14, No. 4 (Aug., 1963), 600-606.

\end{thebibliography}
\end{document}